\documentclass[10pt, align]{amsart}
\usepackage{multicol}
\usepackage{breqn}
\usepackage{wasysym}
\usepackage{bookmark}
\usepackage{soul}

\numberwithin{equation}{section}
\usepackage[a4paper,portrait,left=3.5cm,right=3.5cm,top=3.5cm,bottom=3.5cm]{geometry} \usepackage{amsmath,amsfonts,amscd,amssymb,amsthm,color}
\usepackage{thmtools}
\usepackage[mathscr]{euscript}
\usepackage{cancel} 
\usepackage{verbatim}
\usepackage{hyperref}
\usepackage{todonotes}
\usepackage{enumitem}
\usepackage[tikz]{mdframed}
\usepackage{marginnote}
\usepackage{dsfont}

\usepackage[UKenglish]{babel}
\usepackage[UKenglish]{isodate}
\cleanlookdateon

\usepackage{color}
\definecolor{Red}{cmyk}{0,1,1,0}

\usepackage[capitalise]{cleveref}

\newtheorem{theorem}{Theorem}[section]

\newtheorem{corollary}[theorem]{Corollary}
\newtheorem{lemma}[theorem]{Lemma}

\newtheorem{proposition}[theorem]{Proposition}

\newtheorem{remark}[theorem]{Remark}

\crefname{equation}{}{}
\Crefname{equation}{}{}

\newcommand{\bb}[1]{{\mathbb #1}}

\makeatletter
\def\ubar#1{\underline{\sbox\tw@{$#1$}\dp\tw@\z@\box\tw@}}
\makeatother

\renewcommand\P[2]{\mathbb{P}_{#1}^{{\mathrm{BES}(#2)}}}
\newcommand\PD[3]{\mathbb{P}_{#1}^{{\mathrm{BES}(#3)},#2}}
\newcommand\PB[4]{\mathbb{P}_{#1,#2,#3}^{{\mathrm{BES}(#4)},\mathrm{br}}}

\newcommand\EB[4]{\mathbb{E}_{#1,#2,#3}^{{\mathrm{BES}(#4)},\mathrm{br}}}

\title[Sharp Barrier Estimates for Bessel bridges]{Sharp barrier estimates for Bessel  bridges}
\date{\today}

\author{Leandro Chiarini}
\author{Ellen Powell}
\address{\vspace{-0.5em}Instituto de Matemática e Estatística, Universidade de São Paulo, São Paulo, Brazil}
\address{\vspace{0.5em}Department of Mathematical Sciences, Durham University, Durham, UK}
\email{lchiarini@usp.br}
\email{ellen.g.powell@durham.ac.uk}

\begin{document}
\maketitle

\begin{abstract}
	In this article, we derive precise estimates for the probability that a Bessel bridge  of dimension $d \ge 0$ and end points $x$ and $a+bT-j$ stays below the linear barrier
	$a + bt$ for all $t \in [0,T]$.
	We identify the leading order term as well as the asymptotic error for this probability as $T\to \infty$, depending on $a,b,j,x$.
	We also derive the behaviour of such leading term as we allow $a,j\to \infty$, and obtain precise bounds for all error terms.
	Finally, we establish a complementary result where the linear barrier is perturbed by a {small} concave function.
\end{abstract}

\section{Introduction}

Ballot problems for random walks and their continuum counterparts, \emph{barrier problems}, have a long and rich history in probability theory, serving as foundational tools for understanding constrained stochastic processes.
Applications range from combinatorics and queuing theory to the fine analysis of extremal events in branching processes and random media \cite{banderier2019kernel,takacs1989ballots,bramson1978maximal,sznitman2013brownian}.
In this article, we are interested in deriving precise estimates for the barrier problem associated with a
\emph{Bessel bridge} of dimension $d\ge 0$.

More precisely, we aim to estimate the probability of the event
\begin{equation*}
	\Omega_{a,b,T}
	:=
	\left\{X_t\le a+bt \quad \forall t\in[0,T] \right\}
\end{equation*}
for $a,b>0$, conditionally on the value of $X_T\in (0,a+bT)$.
Since the density of $X_T$ has an explicit and rather simple expression, one can easily deduce from our results the equivalent estimate barrier probabilities for the unconstrained process.

We note that these questions have been extensively studied in the setting of Brownian motion and random walks (e.g. \cite{Addario-Berry_Reed_2008,doob1949heuristic,Durbin_1971}), and one can obtain crude results for Bessel process using change of measure techniques.
However, the level of precision of our results mean that we must work directly with the Bessel processes and make use of special symmetries (\cref{lem:linear-to-flat-barrier}), distributional identities such as \eqref{eq:defEn}, and fine asymptotic analysis (\cref{lem:step2,lem:step3}).

We write $\PB{x}{y}{T}{d}$ for the law of a $d$-dimensional Bessel bridge from $x$ to $y$ in time $T$, see \cref{sec:pre}, and our main result is as follows.

\begin{theorem} \label{thm:linear-bessel}
	Fix $d\ge 0$ and let $\nu=d/2-1$.
	Then for $a,b,j,T>0$, and $x\in [0,a)$, we have
	\begin{equation*}
		\PB{x}{a+bT-j}{T}{d}(\Omega_{a,b,T})
		:=
		\frac{j(a-x)}{T}
		\left(\mathcal{P}_{a,b,x}+\mathcal{E}_{a,b,x,j}(T)\right)+\frac{j}{T}\widehat{\mathcal{P}}_b(x)
	\end{equation*}
	where $\mathcal{E}_{a,b,x,j}(T)\to 0$ pointwise as $T\to \infty$, $\mathcal{P}_{a,b,x}\to 2$ pointwise as $a\to \infty$, and for $I_{\alpha}(\cdot)$ the modified Bessel function of the first kind,
	\begin{equation*}
		\widehat{\mathcal{P}}_b(x)
		:=
		-x\left(\frac{I_{|\nu|+1}(bx)}{I_{|\nu|}(bx)}-1\right)  \to \frac{2|\nu|+1}{2b} \text{ as } x\to \infty.
	\end{equation*}
	Moreover, for any compact $K\subset (0,\infty)$ there exists $C<\infty$ depending only on $K$ and $d$ such that for all $a > \max\{x,1\}$ and $b\in K$
	\begin{equation*}
		|\mathcal{E}_{a,b,x,j}(T)|\,
		\le \, C\frac{\log^2(T)}{T}\max\left\{\frac{j a^2 }{a-x}, \frac{j a^{5/2} }{(a-x)^4}\right\}
		\text{ and }
		|\mathcal{P}_{a,b,x}-2|\, \le\,  \frac{C}{\sqrt{a}}\frac{1}{\min\{a-x,1\}}.
	\end{equation*}
\end{theorem}
\noindent $\mathcal{P}_{a,b,x}$ depends only on $a,b$ and $x$, and is defined in terms of an infinite sum in \eqref{eq:value-constant-P}.

\begin{remark}
	\emph{The error bound in \cref{thm:linear-bessel} may not be optimal, but it cannot vanish, for example, if $a,j,x=O(T)$ which would lead to a barrier probability of order $1$.}
	\emph{The reason for separating the terms $\mathcal{P}_{a,b,x}$ and $\widehat{\mathcal{P}}_b(x)$ is to distinguish between the limiting behaviour when $a,a-x\to \infty$ and when $a\to \infty$ but $a-x$ does not. In the first case only $\mathcal{P}_{a,b,x}$ is relevant, but in the latter, $\widehat{\mathcal{P}}_{b}(x)$ plays a role. The bound on $\mathcal{E}_{a,b,x,j}(T)$ also becomes much simpler in the case when $a-x=O(a)$.}
\end{remark}

In addition to \cref{thm:linear-bessel}, we have a very general non-sharp upper bound.
\begin{proposition}\label{prop:non-sharp}
	For $d\ge 0$ and compact $K\subset (0,\infty)$ there exists $C_{K,d}<\infty$ such that for all $T>0$ and $a,j \in [0,a+bT), x \in [0,a]$
	\begin{equation}\label{eq:thm-non-sharp-bound}
		\sup_{b\in K}
		\PB{x}{a+bT-j}{T}{d}(\Omega_{a,b,T})
		\le C_{K,d}
			{
				\frac{j((a-x)+1)}{T}.
			}
	\end{equation}
\end{proposition}

Finally, we establish an upper bound for the analogous probability when we add a {small} concave function to the linear barrier.

\begin{theorem}\label{thm:concave-bessel}
	Let $d \ge 0$, $a,b,x,j$ be as in Theorem \ref{thm:linear-bessel}, and let $h:[0,\infty) \longrightarrow [0,\infty)$ be an increasing concave function with $h(0)=0$ satisfying $h(t) \le ct^{\gamma}$ for some $\gamma<1/6$. Then letting $t_T:= \min\{ t,  T-t\}$ and $
		\Omega_{a,b,T}^h
		:=
		\left\{X_t\le a+bt+h(t_T) \quad \forall t\in[0,T] \right\}$
	we have that
	\begin{align*}
		\limsup_{T\to \infty} T\PB{x}{a+bT-j}{T}{d}(\Omega_{a,b,T}^h)
		\le
		j(a-x)\left(\mathcal{P}_{a,b,x}+\frac{\widehat{\mathcal{P}}_b(x)}{a-x}\right)
		(1+\delta_{a,b,j,x}),
	\end{align*}
	where $\delta_{a,b,j,x}<\infty$ for all $a,b,j,x$.
	Furthermore, and for fixed $x,b$ we have that  $\delta_{a,b,j,x} \to 0$ as $a \to \infty$, $j \to \infty$.
\end{theorem}

\begin{comment}
The following is a Corollary of the previous two theorems by integrating over $j$ an unit interval.
\begin{corollary}\label{thm:non-bridge}
	Under the same conditions of \cref{thm:linear-bessel}, we further assume that
	\begin{align*}
		(a,b,x,j) \in
		\tilde{\mathcal{S}}^{T}_{K,\alpha}:=
		\left\{(a,b,x,j) \in
		\mathcal{S}^{T}_{K,\alpha},
		a^2,j^2 \le T^{\alpha}
		\right\},
	\end{align*}
	then for the event
	\begin{align*}
		\tilde{\Omega}_{a,b,j,T}
		:=
		\left[X_t\le a+bt \quad \forall t\in[0,T] \, ; \, X_T\in a+bT-j+[-1,0]\right].
	\end{align*}
	we have
	\begin{align*}
		\P{x}{d}(\tilde{\Omega}_{a,b,j,T})
		=
		\frac{e^{-\frac{1}{2}b^2 T}}{T^{1-\nu}}
		\frac{b^{\nu}(e^{b}-1)}{x^{\nu}e^{b(a-j)}}
		F_T(a,b,x)
		\left(j+ r_{b}(j) \right),
	\end{align*}
	where $r_b(j)$ converges to $0$ uniformly in $b \in K$.

	Finally, under the conditions of \cref{thm:concave-bessel}, and assuming $(a,b,x,j) \in $, we have that
	\begin{align*}
		\limsup_{T\to \infty} 		\,
		 &
		\left(\frac{e^{-\frac{1}{2}b^2 T}}{T^{1-\nu}}
		\frac{b^{\nu}(e^{b}-1)}{x^{\nu}e^{b(a-j)}} \right)^{-1}
		\P{x}{d}(X_t\le a+bt + h(t_T)\quad \forall t\in[0,T]; X_T-(a+bT-j) \in [-1,0] )
		\\&  \le F(a,b,x)(1+\delta_{a,b,j,x}),
	\end{align*}
	where for fixed $x,b$ we have that  $\delta_{a,b,j,x} \to 0$ as $a \to \infty$, $j \to \infty$.
\end{corollary}
\end{comment}

Our motivation to study this version of the barrier problem is on their role in understanding the maxima of log-correlated processes.
Consider, for instance, the maximum of a Branching Brownian Motion (BBM) in dimension 1.
To show that the right-most particle $\bar{W}_T$ at time $T$ behaves like $m_T:=c_1T- c_2 \log T$ for appropriate $c_1,c_2>0$, one estimates the probability that a single Brownian motion exceeds $m_T$ at time $T$, but also stays below an affine barrier (possibly perturbed by a concave function) at all smaller times.
This is a reasonable restriction - and captures the correct asymptotic probabilities - since the (much) smaller number of particles at times $t\ll T$ means that the maximum cannot grow too fast earlier times than at the final time $T$ and yields sharp control on the fluctuations of $\bar{W}_T-m_T$, {\cite{roberts2013simple}}.
A similar strategy applies to log-correlated fields such as the planar Gaussian free field (GFF), whose point-wise values are ill-defined.
Instead, one studies circle averages or other local approximations of the field, which behave roughly like Brownian motions and are nearly independent at well-separated points, see {\cite{bolthausen2001entropic,bramson2012tightness}}.
In our case, we are interested in analogous estimates for Bessel processes, which arise in a plethora of contexts.
For us, the main motivation is the connection with extreme behaviour of the local time field for planar Brownian motion.
Here, approximations to the local times are related to zero-dimensional Bessel processes via the Ray–Knight theorems \cite{jego2020planar}.

While barrier estimates for Bessel processes exist (\cite{salminen2011hitting}, for instance), they are not sharp enough to capture, for example, the limiting law of the renormalised maximum of Brownian local times.
Our approach leverages symmetries of Bessel processes to reduce the affine barrier problem to a flat one with a shifted starting point, followed by a careful asymptotic analysis using decompositions of hitting times for Bessel processes into exponential sums.
\vspace{1em}

\paragraph{\bf Acknowledgements}
During this project, the research of LC and EP were supported by UKRI Future Leader’s Fellowship MR/W008513.
We also thank the Hausdorff Institute for Mathematics in Bonn for their hospitality during the programme ``Probabilistic methods in Quantum Field Theory'' funded by the Deutsche Forschungsgemeinschaft under German's Excellence Strategy – EXC-2047/1, 390685813.
Finally, the authors would also like to thank Antoine Jego for insightful discussions during early stages of this article.

\section{Preliminaries} \label{sec:pre}
In this article, we will use the following notation:
\begin{itemize}
	\item $\P{x}{d}$ denotes the law of a $d$-dim.~Bessel process starting from $x$;
	\item $\PD{x}{z}{d}$ denotes the law of a $d$-dim.~Bessel process with drift $z$ starting from $x$;
	\item $\PB{x}{y}{u}{d}$ denotes the law of a $d$-dim.~Bessel bridge of length $u$ from $x$ to $y$.
\end{itemize}
{Whilst the Bessel process and the Bessel bridge are fairly well-known, the Bessel process with drift is less common.
Its density can be written explicitly \cite[$(4.c\uparrow)$]{pitman1981bessel} in a similar way to that of a Bessel process.
	{ The term originates from the fact that, when starting at $0$, it has the same distribution as the norm of a $d$-dimensional Brownian motion with drift $v$, starting at the origin, where $d \in \mathbb{N}$ and $v \neq 0$. }
As we will only briefly use this process for the proof of \cref{lem:linear-to-flat-barrier} through the basic properties listed below, we will not provide its full description.
For more details on these processes, we suggest \cite{lawler2018notes,pitman1982decomposition,pitman1981bessel}, respectively.
}

We will use the notation $f \lesssim g$ if there exists a constant $C>0$ (which may change from line to line) such that $f \le Cg$.
	{
		Likewise, we write $f = O(g)$ if $|f| \lesssim g$ and $f = o(g)$ if $|f|/g$ goes to $0$ in the relevant choice of parameters.
	}

\subsection{Basic properties of Bessel processes}
We will often use the following properties of the measures above.
In the following, we take $G$ to be a functional $G: C([0,u]) \longrightarrow \mathbb{R}$ for some $0 < u \le \infty$.
\begin{enumerate}
	\item[(I)]   For $u>0$, by \cite[Theorem 5.8]{pitman1981bessel}, we have
	      \begin{equation*}
		      \PB{x}{yu}{u}{d}\left(G\left((X_t)_{t\in[0,u]}\right)\right)
		      =
		      \PD{y}{x}{d}\left(G\left((tX_{1/t-1/u})_{t\in [0,u]}\right)\right).
	      \end{equation*}

	\item[(S)] For $c>0$,  by \cite[Proposition 2.4]{lawler2018notes} we have
	      \begin{align*}
		      \P{x}{d}\left(G((X_t)_{t\ge 0})\right)
		       & =
		      \P{x/\sqrt{c}}{d}\left(G((\sqrt{c}X_{t/c})_{t\ge 0}))\right),            \\
		      \PD{x}{z}{d}(G\left((X_t)_{t\ge 0}))\right)
		       & =
		      \PD{x/\sqrt{c}}{\sqrt{c}z}{d}\left(G((\sqrt{c}X_{t/c})_{t\ge 0})\right), \\
		      \PB{x}{y}{u}{d}\left(G((X_t)_{t\in[0,u]}))\right)
		       & =
		      \PB{x/\sqrt{c}}{y/\sqrt{c}}{u/c}{d}\left(G((\sqrt{c}X_{t/c})_{t\in [0,u]}))\right).
	      \end{align*}

	\item[(D)] For $x,y>0$ and $u>0$, by \cite[Section 5]{pitman1982decomposition}, we have
	      \begin{equation*}
		      \P{x}{0}(G((X_t)_{t\in[0,u]}|X_t>0 \text{ for all } t \in [0,u] )
		      =
		      \PB{x}{y}{u}{4}(G((X_t)_{t\in [0,u]})).
	      \end{equation*}
	      and for $d \in (0,2)$,  $
		      \P{x}{d}(G((X_t)_{t\in[0,u]})
		      =
		      \PB{x}{y}{u}{4-d}(G((X_t)_{t\in [0,u]})).
	      $

	\item[(R)] By \cite[Equation (5.b)]{pitman1982decomposition}, the time reversal of a Bessel bridge from $x\to y$ is a Bessel bridge from $y\to x$.
	\item [(P)]
	      Let $\nu=\nu(d):=d/2-1$, from \cite[Proposition 2.5]{lawler2018notes},  for all $x,y,u>0$
	      \begin{equation*}
		      \P{x}{d}(X_u\in dy)=:p^{(d)}(x,y;u)=\frac{y}{u}\left(\frac{y}{x}\right)^\nu \exp\left(-\frac{x^2+y^2}{2u}\right)I_{|\nu|}\left(\frac{xy}{u}\right),
	      \end{equation*}
	      where $I_{|\nu|}(\cdot)$ denotes the modified Bessel function.
\end{enumerate}

\begin{remark} \label{r: def d}
	We will use the duality relation (D) to reduce our problem to the setting $d\ge 2$. As such, we define
	\begin{equation*}
		\mathsf{d}:=\mathsf{d}(d):=
		\begin{cases}
			4-d & \text{for } d\in [0,2), \\
			d   & \text{for }d\ge 2.
		\end{cases}
	\end{equation*}
	In particular, notice that $\mathsf{d}\ge 2$ and $\frac{\mathsf{d}}{2}-1 = |\nu(d)| \ge 0$  for all $d \ge 0$.
\end{remark}

\subsection{Hitting times for Bessel processes}\label{subsec:hitting-times}
A central part of our approach is to rewrite Bessel barrier probabilities in terms of Bessel process hitting times. That is, hitting times of \emph{fixed} levels.
This change of perspective is useful because the distribution of these hitting times can be characterised as sums of independent exponential random variables.
In this section, we collect the exact results required.

We consider the law under $\P{x}{d}$, with $x<1$ and $d>0$, of the first hitting time $\tau$ of $1$.
By \cite{kent1980eigenvalue} (see also \cite[Section 2.2]{bednorz2023moments} for discussion on these results),
we have the decomposition in terms of the a.s.~convergent series
\begin{equation}\label{eq:defEn}
	\tau
	\stackrel{d}{=}
	\sum_{n=1}^\infty I_n E_n,
\end{equation}
where all the random variables in the sum are independent, the $I_n$ are i.i.d  Bernoulli random variables with success probability  $1-x^2$ and $E_n \sim \mathrm{Exp}(j_{|\nu|,n}^2/2)$ for each $n$, where the $j_{|\nu|,n}$ is the $n$-th zero of the Bessel function (of the first kind) $J_{|\nu|}$.
Note that the law of the $E_n$ does not depend on $x$.
We will make use of the following lemma that allow us to obtain convergence of sums of functions over ``rescaled Bessel zeros''.
Its proof is a simple consequence of the asymptotic behaviour of zeros of Bessel functions, that is,
\begin{equation}\label{eq:MacMahon}
	j_{|\nu|,n}	=
	\pi \left(n + \frac{2 |\nu| -1}{2} \right)
	+
	\frac{4\nu^2-1}{2(\pi n + 2\nu-1)}
	+ O(n^{-2}).
\end{equation} \vspace{-1em}
\begin{lemma}\label{lem:weak-convergence}
	For $\eta \in (0,\infty)$, consider the measure
	\begin{equation*}
		\mu^{(\eta)}
		:=
		\frac{1}{\sqrt{\eta} }\sum_{n \ge 1} \delta_{\frac{j_{|\nu|,n}}{\sqrt{\eta}}},
	\end{equation*}
	where $\delta_z$ denotes the Dirac measure at $z$.
	Then there exists $C<\infty$ such that for all $f' \in C^0 ([0,\infty)) \cap L^1([0,\infty))$, we have
	\begin{align*}
		\left|\int f d\mu^{(\eta)}
		-
		\frac{1}{\pi}
		\int_0^\infty f(u) du\right|
		\le \frac{C}{\sqrt{\eta}} \|f'\|_{L^1([0,\infty))}.
	\end{align*}
	Moreover, if $\lim_{x \to \infty} f(x) = f(0)$ and $f'' \in C^0 ([0,\infty)) \cap L^1([0,\infty))$, we have
	\begin{align*}
		\left|\int f d\mu^{(\eta)}
		-
		\frac{1}{\pi}
		\int_0^\infty f(u) du\right|
		\le \frac{C}{\eta} \|f''\|_{L^1([0,\infty))}.
	\end{align*}
\end{lemma}

\section{Proof of linear barrier estimates}

Our strategy to calculate $\PB{x}{a+bT-j}{T}{d}(\Omega_{a,b,T})$ consists of three main steps.
\begin{enumerate}
	\item Convert probability that a Bessel bridge stays below an affine barrier to the probability that a (different) Bessel bridge stays below a certain \emph{flat} barrier.
	\item Obtain an implicit expression for the leading order term as $T \to \infty$.
	\item Obtain asymptotic expressions as we allow $a,j$ to diverge.
\end{enumerate}

\paragraph{\bf Step (1)}
In this step we rewrite the bridge probability $\PB{x}{a+bT-j}{T}{\mathsf{d}}(\Omega_{a,b,T})$ in terms of the probability that a $\mathsf{d}$-dimensional Bessel bridge hits a \emph{constant} barrier.
This calculation is inspired by the proof of \cite[Theorem 17]{salminen2011hitting}.

\begin{lemma}\label{lem:linear-to-flat-barrier}
	For any $T>0$, $a,b>0$ and $x\in [0,a]$ and $j\in (0,a+bT)$, we have that for $\Gamma^{a,b}_T:=\frac{T}{a(a+bT)}$,
	\begin{align*}
		\PB{x}{a+bT-j}{T}{\mathsf{d}}(X_t\le a+bt \quad \forall t\in [0,T])
		=
		\PB{1-\frac{j}{a+bT}}{\frac{x}{a}}{\Gamma^{a,b}_T}{\mathsf{d}}(X_t\le 1 \quad \forall t\in [0,\Gamma^{a,b}_T]).
	\end{align*}
\end{lemma}

Notice that, for fixed $a,b$, $\Gamma^{a,b}_T \to 1/ab$ as $T \to \infty$.
\begin{proof}
	First, we apply the time inversion property (I) and then substitute $u=1/t-1/T$ to write
	\begin{align}
		(\star):=\PB{x}{a+bT-j}{T}{\mathsf{d}}(X_t\le a+bt \quad \forall t\in [0,T])
		 &
		= \,
		\PD{b+\frac{(a-j)}T}{x}{\mathsf{d}}(tX_{1/t-1/T}\le a+bt \quad \forall t\in [0,T])
		\nonumber \\
		 & =
		\PD{\frac{a+bT-j}T}{x}{\mathsf{d}}(X_u\le au + \frac{a+bT}{T} \quad \forall u\ge 0) \nonumber.
	\end{align}
	Next, we use the scaling property (S) with $c=(a+bT)/a$ to obtain that
	\begin{align}
		(\star) =\, & \PD{\frac{\sqrt{a(a+bT)}}T(1-\frac{j}{a+bT})}{\sqrt{a(a+bT)}\frac{x}{a}}{\mathsf{d}}\left(
		\frac{\sqrt{a+bT}}{\sqrt{a}}X_{\frac{au}{a+bT}}\le au + \frac{a+bT}{T} \quad \forall u\ge 0\right) \nonumber                                                                        \\
		= \,        & \PD{\frac{\sqrt{a(a+bT)}}T(1-\frac{j}{a+bT})}{\sqrt{a(a+bT)}\frac{x}{a}}{\mathsf{d}}\left(X_{v}\le \sqrt{a(a+bT)} \big(v+\frac{1}{T}\big),\quad \forall v\ge 0\right)
		\nonumber
	\end{align}
	where the second equality follows by substituting $v={au}/(a+bT)$.
	From here we use time inversion again, first substituting $v=1/r-1/T$, to further rewrite $(\star)$ as
	\begin{align}
		(\star) =\, & \PD{\frac{\sqrt{a(a+bT)}}T(1-\frac{j}{a+bT})}{\sqrt{a(a+bT)}\frac{x}{a}}{\mathsf{d}}\left(\frac{r}{r}X_{\frac{1}{r}-\frac{1}{T}} \le \sqrt{a(a+bT)} \frac{1}{r}\quad \forall r\in [0,T]\right)
		\nonumber                                                                                                                                                                                                    \\
		=\,
		            & \PB{\sqrt{a(a+bT)}\frac{x}{a}}{\sqrt{a(a+bT)}(1-\frac{j}{a+bT})}{T}{\mathsf{d}}(X_t\le \sqrt{a(a+bT)} \quad \forall t\in [0,T]).
		\nonumber
	\end{align}
	The final step is to use scaling (S) again, and a simple time reversal, property (R), to get
	\begin{align}
		(\star) & =\,
		\PB{\frac{x}{a}}{1-\frac{j}{a+bT}}{\frac{T}{a(a+bT)}}{\mathsf{d}}\left(X_t\le 1 \quad \forall t\in \left[0,\frac{T}{a(a+bT)}\right]\right)
		\nonumber     \\
		        & =\,
		\PB{1-\frac{j}{a+bT}}{\frac{x}{a}}{\frac{T}{a(a+bT)}}{\mathsf{d}}\left(X_t\le 1 \quad \forall t\in \left[0,\frac{T}{a(a+bT)}\right]\right).
		\nonumber
	\end{align}
	To conclude, we simply use the definition of $\Gamma^{a,b}_T$.
\end{proof}

Steps (2) and (3) can be summarised, respectively, by the following two results.

\begin{lemma} \label{lem:step2}
	For $a,b,j,x >0$, such that $j < a+bT$ and $x<a$, we have
	\begin{equation*}
		\PB{1-\frac{j}{a+bT}}{\frac{x}{a}}{\frac{T}{a(a+bT)}}{{d}}
		\left(
		X_t\le 1 \, \forall t\in \left[0,\tfrac{T}{a(a+bT)}
			\right]\right)= \\
		{
		\frac{j(a-x)}{T}\big( \mathcal{P}_{a,b,x}+\mathcal{E}_{a,b,x,j}(T)\big)
		+
		\frac{ j}{T}
		\widehat{\mathcal{P}}_b(x)
		},
	\end{equation*}
	where for $E_n$ defined as in \eqref{eq:defEn},
	\begin{align}\label{eq:value-constant-P}
		\mathcal{P}_{a,b,x}
		 & :=
		1+
		\frac{2}{b(a-x)}\sum_{n\ge 1}
		\left(
		1-
		\mathbb{E} \left[
			\frac{p^{(\mathsf{d})}(1,\frac{x}{a},\frac{1}{ab}-E_n)}{p^{(\mathsf{d})}(1,\frac{x}{a},\frac{1}{ab})}
			\mathds{1}_{E_n\le \frac{1}{ab}}\right]
		\right),
	\end{align}
	and for any compact $K\subset (0,\infty)$, there exists $C_{K,\mathsf{d}}<\infty$ depending only on $K$ and $\mathsf{d}$ such that
	\begin{equation} \label{eq:cond-for-g}
		\sup_{b\in K} |\mathcal{E}_{a,b,x,j}(T)|
		\, \le \,
		C_{K,\mathsf{d}}
		\frac{j a^2}{T(a-x)}\log^2(T) \left(1+ \frac{\sqrt{a}}{(a-x)^3}\right).
	\end{equation}
	In particular,  the uniform bounds on $|\mathcal{E}_{a,b,x,j}(T)|$ in \cref{thm:linear-bessel} and in
	\cref{eq:thm-non-sharp-bound} hold.
\end{lemma}

\begin{lemma}\label{lem:step3}
	For any $K\subset (0,\infty)$ compact there exists $C_{K,\mathsf{d}}<\infty$ depending only on $K$ and $\mathsf{d}$ such that
	\begin{equation*}
		\sup_{b\in K, a\ge 1\vee x}
		\left|
		\frac{2}{b(a-x)}\sum_{n\ge 1}
		\left(
		1-
		\mathbb{E} \left[
			\frac{p^{(\mathsf{d})}(1,\frac{x}{a},\frac{1}{ab}-E_n)}{p^{(\mathsf{d})}(1,\frac{x}{a},\frac{1}{ab})}
			\mathds{1}_{E_n\le \frac{1}{ab}}\right]
		\right)
		-1\right| \leq C_{K,\mathsf{d}} \frac{1}{\sqrt{a}}\frac{1}{\min\{a-x,1\}}.
	\end{equation*}

\end{lemma}

Before we prove either of these lemmas, we will conclude the proof of the theorem assuming that both hold.

\begin{proof}[Proof of \cref{thm:linear-bessel}]
	Thanks to \cref{lem:linear-to-flat-barrier,lem:step2,lem:step3}, all that remains is to check the behaviour of $\widehat{\mathcal{P}}_b(x)$.
	To obtain 
	the limiting behaviour as $x \to \infty$, we use the asymptotic expressions for $\alpha>0$ (see e.g.~\cite{DLMF})
	\begin{equation}\label{eq:Bessel-asymp-small-z}
		I_{\alpha}(z)
		\sim
		\frac{
			z^{\alpha}
		}{
			2^{\alpha}\Gamma(1+\alpha)
		}
		\text{ as }
		z \to 0,
		\text{ and }
		I_{\alpha}(z)
		=
		\frac{e^z}{\sqrt{2\pi z} }
		\left(
		1
		-
		\frac{4\alpha^2-1}{8z}
		+
		O(z^{-2})
		\right)
	\end{equation}
	as $z\to \infty$,
	to deduce that for $C_{K,\mathsf{d}}>0$ depending only on $K,\mathsf{d}$,
	\begin{align*}
		\lim_{x	\to \infty}
		\sup_{b \in K}
		\left|
		\widehat{\mathcal{P}}_{b}(x)
		-
		\frac{2|\nu|+1}{2b}
		\right|
		\le
		\lim_{x	\to \infty}
		\frac{C_{K,\mathsf{d}}}{x} =0.
	\end{align*}

	\vspace{-1em}
\end{proof}

\paragraph{\bf Step (2)}

Throughout this step, we fix $ \mathsf{d}\ge 2$ and $U\in (0,\infty)$, and
the implicit constants in the notation $a\lesssim b$ will only depend on $\mathsf{d},U$.

\begin{lemma}\label{lem:a1}
	Let $\mathsf{d}\ge2$ and $U\in (0,\infty)$.
	Then for all $\delta,y\in (0,1)$ and $U>u>\delta$:
	\begin{equation*}
		\frac{p^{(\mathsf{d})}(1,y,u)}{p^{(\mathsf{d})}(1-\delta,y,u)} \ge 1-\frac{\delta(1-y)}{u}-3\delta|\nu|.
	\end{equation*}

	Moreover, for all $\delta,y\in (0,1)$ and $u>\delta$:
	\begin{equation*}
		\left|
		1
		-
		\frac{p^{(\mathsf{d})}(1,y,u)}{p^{(\mathsf{d})}(1-\delta,y,u)}
		-
		\frac{\delta}{u}
		\left(1
		-
		y\frac{I_{|\nu|+1}(\frac{y}{u})}{I_{|\nu|}(\frac{y}{u})}\right)
		\right|
		\lesssim
		\frac{\delta^2}{u^2}.
	\end{equation*}
\end{lemma}

\begin{proof}
	By property (P), we have the explicit formula
	\begin{align*}
		\frac{p^{(\mathsf{d})}(1,y,u)}{p^{(\mathsf{d})}(1-\delta,y,u)}
		 & =
		e^{-\frac{\delta(1-\delta/2)}{u}}
		(1-\delta)^{|\nu|}
		\frac{I_{|\nu|}(y/u)}{I_{|\nu|}(y(1-\delta)/u))}
		=
		\frac{
		\big({1-\frac{\delta}{u}+ O(\frac{\delta^2}{u^2}) }\big)
		\left(1- |\nu|\delta + O(\delta^2)\right)
		}{
1
		-
		\frac{I'_{|\nu|}(y/u)}{I_{|\nu|}(y/u)} \frac{\delta y}{u}
		+
		\frac{R_{|\nu|}(y,u)}{I_{|\nu|}(y/u)}
		}.
	\end{align*}
	The desired bounds follow by using the Bessel recurrence relation
	\begin{align}\label{eq:multi-thm}
		I_{|\nu|}'(z)
		=
		I_{|\nu|+1}(z)+\frac{|\nu|}{z} I_{|\nu|}(z),
		\qquad
		I_{|\nu|}''(z)
		=
		\left(1+\frac{|\nu|(|\nu|-1)}{z^2}\right) I_{|\nu|}(z)-\frac{1}{z} I_{|\nu|+1}(z)
	\end{align}
	see e.g \cite[(10.29.2)]{DLMF} to bound error term $R_{|\nu|}(y,u)$,  using the Taylor expansion of $I_{|\nu|}$.
\end{proof}
Recall from \cref{subsec:hitting-times} that under the law of a Bessel process starting from $1-\delta$, the first hitting time $\tau$ of one is equal in distribution to $\tau^{(\delta)}$ defined by
\begin{equation*}
	\tau^{(\delta)}:=\sum_{n\ge 1} I_n^{(\delta)} E_n \,;
	\quad E_n \sim \text{Exp}\left( \frac12{j^2_{|\nu|,n}}\right) \, ;
	\quad I_n^{(\delta)}\sim \mathrm{Ber}(p_\delta)
\end{equation*}
with $I_1^{(\delta)},E_1,I_2^{(\delta)},E_2,...$ independent, $j_{|\nu|,n}$ the $n$-th zero of the Bessel function $J_{|\nu|}$, and $p_\delta:=1-(1-\delta)^2=2\delta(1-\delta/2)$.
In what follows, for compactness, we use the generic notation $\mathbb{P},\mathbb{E}$: the distribution of each random variable being considered will be clear from its symbol or the context.

\begin{lemma} \label{cor:app} 	Let $\mathsf{d}\ge2$ and $U\in (0,\infty)$.
	Then
	\begin{align*}
		\PB{1-\delta}{y}{u}{{\mathsf{d}}}(X_t\le 1 \, \forall t\in [0,u])=
		\frac{\delta(1-y)}{u} \left(
		\begin{array}{c}
				1-\frac{y}{(1-y)}\left(\frac{I_{|\nu|+1}(\frac{y}{u})}{I_{|\nu|}(\frac{y}{u})}-1\right)+
				\frac{2u}{1-y}\sum_{n\ge 1}e^{-\frac{j_{|\nu|,n}^2}{2}u}
				\\[1em]
				+
				\frac{2u}{1-y}\sum_{n\ge 1}\mathbb{E}\left[\left(1-\frac{p^{(\mathsf{d})}(1,y,u-E_n)}{p^{(\mathsf{d})}(1,y,u)}\right)\mathds{1}_{E_n\le u}\right]
				+
				\varepsilon_{\delta,u,y}
			\end{array}
		\right)
	\end{align*}
	where	for all $0<\delta,y<1$, and $\delta<u<U$,
	\begin{equation}\label{eq: bound eps dyu}
		\varepsilon_{\delta,u,y}
		\lesssim
		\delta |\log^2(\delta)|
{
		\big(\frac{1}{u(1-y)}+\frac{u^{3/2}}{(1-y)^4}\big).
		}
	\end{equation}
	Moreover, for the same set of $\delta, y, u$ we have
	\begin{equation*}\PB{1-\delta}{y}{u}{{\mathsf{d}}}(X_t\le 1 \, \forall t\in [0,u]) \lesssim \frac{\delta(1-y)}{u}.\end{equation*}
\end{lemma}

\begin{proof}

	We write
	\begin{align}\label{eq:twoterms}\PB{1-\delta}{y}{u}{{\mathsf{d}}}(X_t\le 1 \, \forall t\in [0,u])
		 & = 1-\frac{p^{(\mathsf{d})}(1,y,u)}{p^{(\mathsf{d})}(1-\delta,y,u)}\mathbb{E}\left[f_{u,y}(\tau^{(\delta)}) \mathds{1}_{\tau^{(\delta)}\le u}\right]
	\end{align}
	with $f_{u,y}(x)=\frac{p^{(\mathsf{d})}(1,y,u-x)}{p^{(\mathsf{d})}(1,y,u)}$. By \cref{lem:a1}
	\begin{equation*}
		\frac{p^{(\mathsf{d})}(1,y,u)}{p^{(\mathsf{d})}(1-\delta,y,u)}
		=
		1-\frac{\delta(1-y)}{u}
		\left(1-\frac{y}{1-y}\left(\frac{I_{|\nu|+1}(\frac{y}{u})}{I_{|\nu|}(\frac{y}{u})}-1\right) + \mathsf{Error}\right)
	\end{equation*}
	where $\mathsf{Error}\lesssim\frac{\delta}{u(1-y)}$. Thus, comparing the lemma statement with \eqref{eq:twoterms}, it remains to show that
	\begin{equation}\label{eq:fyuaim}
		\mathbb{E}\left[f_{u,y}(\tau^{(\delta)}) \mathds{1}_{\tau^{(\delta)}\le u}\right]=1{-}2\delta \sum_{n\ge 1} e^{-\frac{j_{|\nu|,n}^2}{2}u}{-}2\delta\sum_{n\ge 1} \mathbb{E}\big[(1-f_{u,y}(E_n))\mathds{1}_{E_n\le u}\big]+\mathsf{Error}
	\end{equation}
	with $\mathsf{Error}\lesssim \delta^2|\log^2(\delta)|(u^{-2}+u^{1/2}(1-y)^{-3}).$ 

	By decomposing according to the first $n$ with $I_n^{(\delta)}=1$, and writing $\tau_n^{(\delta)}=\sum_{k>n} I_k^{(\delta)}E_k$, $q_\delta=(1-p_\delta)$, we have
	\begin{align*}
		 & \mathbb{E}
		\big[
		f_{u,y}(\tau^{(\delta)})
		\mathds{1}_{\{\tau^{(\delta)}\le u\}}
		\big]
		=
		p_\delta \sum_{n\ge 1}q_\delta^{n-1}
		\mathbb{E}\big[
		f_{u,y}(E_n+\tau^{(\delta)}_n)
		\mathds{1}_{E_n+\tau_n^{(\delta)}\le u}
		\big]
		\\ & =
		1-p_\delta
		\sum_{n\ge 1}q_\delta^{n-1}
		\mathbb{P}\big(E_n+\tau_n^{(\delta)} \ge u\big)
		+
		p_\delta
		\sum_{n\ge 1}q_\delta^{n-1}
		\mathbb{E}\big[ \big(
		f_{u,y}(E_n+\tau_n^{(\delta)})-1
		\big)
		\mathds{1}_{E_n+\tau_n^{(\delta)}\le u\}}
		\big].
	\end{align*}
Using the asymptotic expressions from \cref{eq:Bessel-asymp-small-z}, we can derive that
	\begin{align*}\sup_{x\in [0,u)}f_{u,y}(x)
			      & \lesssim
			      \sup_{x\in [0,u)}
			      \max \left\{
			      \frac{\sqrt{u}}{\sqrt{u-x}}e^{-\frac{(1-y)^2}{2}\left(\frac{1}{u-x}-\frac{1}{u}\right)},
			      1 \right\}
			      \lesssim
			      \max \left\{\frac{\sqrt{u}}{(1-y)}, 1 \right\} 
    \end{align*}
    and 
    \begin{align*}
		\sup_{x\in [0,u)}
			      |f'_{u,y}(x)|
			      & \lesssim
			      \sup_{x\in [0,u)}
			      \max
			      \left\{1,
			      \frac{\sqrt{u}}{\sqrt{u-x}}
			      e^{-\frac{(1-y)^2}{2}\left(\frac{1}{u-x}-\frac{1}{u}\right)}
			      \left(\frac{1}{u-x}
			      +
			      \frac{(1-y)^2}{2(u-x)^2}\right)
			      \right\} \\
		     & \lesssim
			      \max
			      \left\{
			      \frac{\sqrt{u}}{(1-y)^3},1
			      \right\}.
\end{align*}
In addition, for $\lambda <\frac12 j_{\nu,n+1}^2$,
		      $\mathbb{E}[
				      e^{\lambda \tau_n^{(\delta)}}
			      ]-1
			      =
			      \exp\big(\sum_{m>n} \log(1+p_\delta \frac{2\lambda}{j_{\nu,m}^2-2\lambda})\big)-1
			      \lesssim
			      \frac{\delta \lambda \log n}{n}$.

	With these observations, we bound
	\begin{align*}
		 & \left| \, p_\delta\sum_{n\ge 1}q_\delta^{n-1}
		\mathbb{E}\left[(f_{u,y}(E_n+\tau_n^{(\delta)})-1)\mathds{1}_{E_n+\tau_n^{(\delta)}\le u}\right]
		-
		2\delta \sum_{n\ge 1}
		\mathbb{E}\left[
			(f_{u,y}(E_n)-1)\mathds{1}_{\{E_n\le u\}}
			\right]\right|
		\\
		 &
		= \sum_{n\ge 1}\left(2\delta-p_\delta q_\delta^{n-1}\right)\,
		\mathbb{E}\left[ \left|
			(f_{u,y}(E_n)-1)\mathds{1}_{\{E_n\le u\}}
			\right| \right]
		\\
		 &
		+
		p_\delta
		\sum_{n\ge 1} q_\delta^{n-1}
		\left|
		\mathbb{E}\left[
		\left(f_{u,y}\left(E_n+\tau_n^{(\delta)}\right)-1\right)\mathds{1}_{E_n+\tau_n^{(\delta)}\le u}\right]-\mathbb{E}[(f_{u,y}(E_n)-1)\mathds{1}_{E_n\le u}]
		\right|.
\end{align*}
	The first term on the right is
	\begin{align*}
		\lesssim
		\delta \sum_{n\ge 1} \min\{\delta n,1\} \mathbb{E}[|f(E_n)-1|]
		\lesssim
		\delta \sup_{[0,u]}|f_{u,y}'| \sum_{n\ge 1} \frac{\min\{\delta n,1\} }{n^2}
		 & \lesssim \delta ^2 |\log \delta|
		\max \left\{
		\frac{\sqrt{u}}{(1-y)^3},1\right\}
	\end{align*}
	and the second is
	$	\lesssim \delta
		\sum_{n\ge 1}
		\mathbb{E}\left[e^{\lambda_n \tau_n^{(\delta)}}+\tau_n^{(\delta)}\right]
		\mathbb{E}\left[|f(E_n)-1|\right]$ for $\lambda_n:= j^2_{|\nu|,n}/2$, which in turn is
	\begin{align*}
		\lesssim
		\sup_{[0,u]}|f_{u,y}'|
		\sum_{n\ge 1} \frac{1}{n^2} \min\{n \delta \, |\log \delta|,1\}
		\lesssim \delta ^2 |\log^2 \delta| \max\left\{\frac{\sqrt{u}}{(1-y)^3},1\right\},
	\end{align*}
	In a similar manner we can show that
	\begin{align*}
		\big | \, p_\delta \sum_{n\ge 1}q_\delta^{n-1}\mathbb{P}(E_n+\tau_n^{(\delta)} \ge u) - p_\delta \sum_{n\ge 1}\mathbb{P}(E_n \ge u) \, \big| \lesssim \frac{\delta^2}{u}|\log u|
	\end{align*}
which yields \eqref{eq:twoterms} and thus the first claim of the lemma.

	For the final claim of the lemma, \eqref{eq: bound eps dyu},  we first show that for $z<1$ and $v<U$:
	\begin{equation}\label{eq:upintermediate}
		\PB{1-\delta}{z}{v}{{\mathsf{d}}}(X_t\le 1 \, \forall t\in [0,v]) \lesssim \frac{\delta}{\sqrt{v}}\left(1+\frac{(1-z)^2}{v}\right).
	\end{equation}
	To see this, we note that
	$
		\mathbb{E}[f_{v,z}(\tau^{(\delta)})\mathds{1}_{\{\tau^{(\delta)}\le v\}}]\ge \mathbb{E}[f_{v,z}(\tau^{(\delta)})\mathds{1}_{\{\tau^{(\delta)}\le v/2\}}],
	$
	and that there exists $C<\infty$ (depending only on $\mathsf{d},U$), such that on the event that $\tau^{(\delta)}\le v/2$, $f_{v,z}(\tau^{(\delta)})\ge 1-C\tau^{(\delta)}\frac{(1-z)^2}{v}$ for all $z<1$ and $v<U$. Taking expectations and using \cref{lem:a1} plus the argument from the first part of this proof - which shows that $\mathbb{P}(\tau^{(\delta)}\ge v/2)\lesssim \delta/\sqrt{v}$ - completes the proof of \eqref{eq:upintermediate}.

	Now to bound $\PB{1-\delta}{y}{u}{\mathsf{d}}(X_t \le 1, \forall t\in [0,u])$ we condition on the value at time $u/2$ to obtain
	\[
		\PB{1-\delta}{y}{u}{\mathsf{d}}(X_t \le 1 \forall t\in [0,u]) \lesssim \frac{\delta(1-y)}{u}\EB{1-\delta}{y}{u}{\mathsf{d}}\left(\frac{2(1-X_{u/2})^2}{u}\right).
	\]
	It is straightforward to see that the expectation on the right is $\lesssim 1$, by scaling and using simple bounds for the Radon--Nikodym derivative between $X_{u/2}$ under the law of $\PB{1-\delta}{y}{u}{\mathsf{d}}$ and under the (non-bridge) Bessel law started from $1-\delta$.
\end{proof}

From \cref{cor:app}, we easily obtain \cref{lem:step2}.

\begin{proof}[Proof of \cref{lem:step2}]
	This follows from \cref{cor:app} with $\delta=j/(a+bT)$, $u_T={T}/{a(a+bT)}$ and $y=x/a$,
	{rewriting
	\[ \frac{2u}{1-y}\sum_{n\ge 1} e^{-\frac{j_{|\nu|,n}^2}{2}u}+\frac{2u}{1-y}\sum_{n\ge 1} \mathbb{E}\big[(1-f_{u,y}(E_n))\mathds{1}_{E_n\le u}\big]=\frac{2u}{1-y}\sum_{n\ge 1} \big(1-\mathbb{E}\big[f_{u,y}(E_n)\mathds{1}_{E_n\le u}\big]\big)\]
	and also bounding the error created by replacing $u_T$ with $u=1/ab$.}
	(This last step follows from the same arguments used in the proof of \cref{cor:app}, using that $u_T\in [u-\frac{1}{b^2T},u]$, so we leave the details to the reader.)
\end{proof}

\paragraph{\bf Step (3)} Finally, we turn to the proof of \cref{lem:step3}
\begin{proof}[Proof of \cref{lem:step3}]
	For the proof, we fix $K\subset (0,\infty)$ compact: whenever we use the notation $O(f(a))$, we mean a term that is bounded in modulus by a constant times $f(a)$, uniformly over $b\in K$, $x\in [0,a]$, and $a\ge 1$ (say). We analyse the series
	\begin{align*}
		\frac{2}{b(a-x)}\sum_{n\ge 1}
		\left(
		1-\mathbb{E} \big[
			\frac{p^{(\mathsf{d})}(1,\frac{x}{a},\frac{1}{ab}-E_n)}{p^{(\mathsf{d})}(1,\frac{x}{a},\frac{1}{ab})}
			\mathds{1}_{E_n\le \frac{1}{ab}}\big]\right)= \frac{C_{a,b,x}}{a-x}+\frac{C'_{a,b}}{a-x}
	\end{align*}
	where
\begin{equation*}\label{eq:value-constant-C}
		C_{a,b,x}
		:=
		\frac{2 }{b}\sum_{n\ge 1}\left(
\frac{j_{|\nu|,n}^2}{2}
		\int_0^{1/ab} e^{-\frac{j_{|\nu|,n}^2}{2}y}
		\frac{p^{(\mathsf{d})}\left(1,\frac{x}{a}; \frac{1}{ab}\right)-p^{(\mathsf{d})}\left(1,\frac{x}{a}; \frac{1}{ab}-y\right)
		}{
		p^{(\mathsf{d})}\left(1,\frac{x}{a}; \frac{1}{ab}\right)}\right)  dy
	\end{equation*}
	and for $f(y):=e^{-\frac{y^2}{2b}}$, applying \cref{lem:weak-convergence},
\begin{align} \label{eq:c1-conv}
		{C'_{a,b}}=\frac{2}{b}\sum_{n\ge 1}e^{-\frac{j_{|\nu|,n}^2}{2ab}}
		=
		\frac{2\sqrt{a}}{b}
		\int_0^\infty
		f
		d\mu^{(a)}
		=
		\sqrt{\frac{2a }{\pi b }}
		+ O\left(\frac{1}{\sqrt{a} }\right).
	\end{align}

We will show that 
    	\begin{align} \label{eq:c2-conv}
		{C_{a,b,x}}=(a-x)
		-\sqrt{\frac{2a }{\pi b }}
		+ O\left(\frac{1}{\sqrt{a} }\right) \max(a-x,1)
	\end{align}
    which, together with \eqref{eq:c1-conv}, completes the proof.

    To show \eqref{eq:c2-conv} we substitute $w=aby/(1-aby)$ and let$\hat{f}(z)=\frac{z^2}{2}e^{-z^2/2}$ to rewrite
    \begin{align}\label{eq:c2-abx}
		C_{a,b,x}
		 & =
		{2}\sqrt{\frac{a}{b}}\sum_{n\ge 1}
		\int_0^{\infty} \frac{\hat{f}\left(j_{|\nu|,n}/\sqrt{\eta(a,b,w)}\right)}{\sqrt{\eta(a,b,w)}}\, g(a,b,x,w) dw 
        \end{align}
        where $\eta(a,b,w)=\frac{ab(1+w)}{w}\ge ab\to \infty$ as $a\to \infty$ and
        \begin{align*}
        g(a,b,x,w) & =\frac{1}{w^{3/2}\sqrt{1+w}}
		\left( 1-\left(1+w\right){e^{-\frac{abw}{2}\left(1+\frac{x^2}{a^2}\right)}}\frac{I_{|\nu|}\left(xb\left(1+w\right)\right)}{I_{|\nu|}(xb)}\right)\nonumber \\
       & =\frac{1}{w^{3/2}\sqrt{1+w}}
		\left( 1-\sqrt{1+w} \, {e^{-\frac{bw}{2a}(a-x)^2}}\right) + \frac{{e^{-\frac{bw}{2a}(a-x)^2}}}{w^{3/2}}R(b,x,w)\nonumber \\
 \text{ with }
    R(b,x,w) & =1-\sqrt{1+w}e^{-bwx}\frac{I_{|\nu|}\left(xb\left(1+w\right)\right)}{I_{|\nu|}(xb)}.
    \end{align*}
    It is straightforward to verify that for $C<\infty$ (depending only on $\mathsf{d},K$)  we have 
    \begin{equation}\label{eq:boundR}|R(b,x,w)|\le \frac{C}{\max(bx,1)}\min\{w,1\}
    \end{equation}   
    for all {$w,x\ge 0, b\in K$}. Moreover, by \cref{lem:weak-convergence}
    \[ 
    \sum_{n\ge 1} \frac{\hat{f}\left(j_{|\nu|,n}/\sqrt{\eta(a,b,w)}\right)}{\sqrt{\eta(a,b,w)}}= \frac1{2\sqrt{2\pi}}+O\left(\frac{1}{a}\right).
    \]
   Using \eqref{eq:boundR}, it is clear that $g(a,b,x,w)$ is integrable over $w$, so we can use Fubini to bring the sum inside the integral. We deduce that
   \[
   C_{a,b,x} = \sqrt{\frac{a}{2\pi b}}\int_0^\infty \left(1+O\left(\frac{1}{a}\right)\right)\left( \frac{1-\sqrt{1+w} \, {e^{-\frac{bw}{2a}(a-x)^2}}}{w^{3/2}\sqrt{1+w}}
		 + \frac{{e^{-\frac{bw}{2a}(a-x)^2}}}{w^{3/2}}R(b,x,w)\right) dw .
   \]
   Now, we can explicitly calculate 
   \begin{equation}\label{eq:explicit1}
\frac{\sqrt{a}}{\sqrt{2\pi b}} \int_0^\infty \frac{dw}{w^{3/2}\sqrt{1+w}}
		\left( 1-\sqrt{1+w} \, {e^{-\frac{bw}{2a}(a-x)^2}}\right)dw = -\sqrt{\frac{2a}{\pi b}}+(a-x).
   \end{equation}
It is also straightforward to verify, using \eqref{eq:boundR}, that 
 \begin{equation}\label{eq:boundR2}
  \frac{\sqrt{a}}{\sqrt{2\pi b}} \int_0^\infty \left(1+O\left(\frac{1}{a}\right)\right)\frac{{e^{-\frac{bw}{2a}(a-x)^2}}}{w^{3/2}}|R(b,x,w)| dw =O\left(\frac{1}{\sqrt{a}}\right) \max(a-x,1) .
   \end{equation}
  Indeed, the left hand side is bounded for all $b\in K$ and $x\in [a/2,a]$ by  $C\sqrt{a}x^{-1}\int_0^\infty \frac{\min(w,1)}{w^{3/2}}\newline=C'/\sqrt{a}$, with $C,C'$ depending only on $K,\mathsf{d}$. On the other hand, for $x\in [0,a/2]$, the left hand side is bounded by $C \sqrt{a} \int_0^\infty \min(w^{-1/2},w^{-3/2}) e^{-\frac{bw}{2a}(a-x)}$ which is in turn less than $C'\sqrt{a-x}\le \sqrt{2}C'(a-x)/{\sqrt{a}}$. Again $C,C'$ do not depend on $x\in [0,a/2]$, $a\ge 1$ or $b\in K$.
Inserting \eqref{eq:explicit1} and \eqref{eq:boundR2} into the expression for $C_{a,b,x}$ completes the proof of \eqref{eq:c2-conv}, and; therefore, the proof of the lemma.
\end{proof}
 
\section{Proof of \cref{thm:concave-bessel}}\label{sec:proof-concave}
We now turn to the proof of \cref{thm:concave-bessel}, which is inspired by the proof of \cite[Equation (11)]{bramson2016convergence}. 
Remember that $h$ is positive, concave and $h(t) \le c t^{\gamma}$ for some $\gamma < 1/6$. 
Recalling the notation $\Omega=\Omega_{a,b,T}$ for the event that $X$ stays below $a+bt$ on $[0,T]$, we further consider the event
\begin{equation*}
	\Omega^h
	=
	\Omega^h_{a,b,T}
	:=
	\left\{X_t\le a +bt + h(t_T)\quad \forall t\in[0,T] \right\}.
\end{equation*}
We want to show that for $a,b,j>0$ and $x\in [0,a)$
\begin{align*}
		\limsup_{T\to \infty} T\PB{x}{a+bT-j}{T}{d}(\Omega_{a,b,T}^h)
		\le
		j(a-x)\big(\mathcal{P}_{a,b,x}+\frac{\widehat{\mathcal{P}}_b(x)}{a-x}\big)
		(1+\delta_{a,b,j,x}),
	\end{align*}
	with $\delta_{a,b,x,j}<\infty$ for fixed $a,b,x,j$ and  for fixed $x,b$ we have that  $\delta_{a,b,j,x} \to 0$ as $a,j\to \infty$.

Clearly we have $\Omega  \subset \Omega^{h}$, so we aim to control the probability of $\Omega^{h}\setminus \Omega$.
Letting $Y_t:=X_t-bt$, the main idea of the proof is to decompose the event $\Omega^{h}\setminus \Omega$ into separate events for path of $X$ before and after time
\begin{equation*}
	\tau:= \inf \big\{t \in [0,T]: Y_t = \sup\nolimits_{s \in [0,T]} Y_s \big\},
\end{equation*}
whose probabilities can be controlled using \cref{thm:linear-bessel} after conditioning on $\tau, X_\tau$.

Let $\mathtt{Bulk}=\mathtt{Bulk}_{a,x,j,T} := [(a-x)^{1/2},T-j^{1/2}]$, which is simply a region far from the end points of $[0,T]$.
(The value ${1/2}$ is not special, we merely choose a small power).
Trivially, we have that
\begin{align*}
	 & \PB{x}{a+bT-j}{T}{d}{(\Omega)}
	\le
	\PB{x}{a+bT-j}{T}{d}{(\Omega^h)}
	\\&  =
	\PB{x}{a+bT-j}{T}{d}(\Omega)
	+
	\PB{x}{a+bT-j}{T}{d}(\Omega^h\setminus \Omega, \tau \not \in \mathtt{Bulk})
	+
	\PB{x}{a+bT-j}{T}{d}(\Omega^h\setminus \Omega, \tau  \in \mathtt{Bulk}).
\end{align*}

Using \cref{thm:linear-bessel}, and since we always keep $x,b$ fixed, \cref{thm:concave-bessel} follows from the following lemma. 
\begin{lemma} \label{lem:c-bulk-bdy}
	Under the conditions of \cref{thm:concave-bessel}, we have
	\begin{enumerate}
		\item[(i)] \label{lem:c-bulk}
		      $
			      \mathcal{L}^{\mathtt{Bulk}}_{a,b,j,x}
			      :=
				  {\displaystyle \limsup_{T \to \infty}}
			      T
			      \PB{x}{a+bT-j}{T}{d}(\Omega^h\setminus \Omega, \tau  \in \mathtt{Bulk})
			      <\infty$
		      {\&} $
				  {\displaystyle\lim_{ a,j \to \infty}}
			      \frac{\mathcal{L}^{\mathtt{Bulk}}_{a,b,j,x}}{aj}
			      =
			      0.
		      $

		\item [(ii)]\label{lem:c-bulk-c}
		      $
			      \mathcal{L}^{(\mathtt{Bulk})^c}_{a,b,j,x}
			      :=
				  {\displaystyle \limsup_{T \to \infty}}
			      T
			      \PB{x}{a+bT-j}{T}{d}(\Omega^h\setminus \Omega, \tau \not \in \mathtt{Bulk})
			      <\infty$
		      {\&} $
				  {\displaystyle\lim_{ a,j \to \infty}}
			      \frac{\mathcal{L}^{\mathtt{(Bulk)^c}}_{a,b,j,x}}{aj}
			      =
			      0.
		      $
	\end{enumerate}
\end{lemma}

Throughout this proof, the omitted constants in $\lesssim$ depend only in $d,K,x$ and $\gamma$. 

\begin{proof}[Proof of \cref{lem:c-bulk-bdy}(i)]
	We prove this lemma by evaluating $\bb P(\Omega^h\setminus \Omega, \tau\in [k,k+1])$ for appropriate $k$'s, then summing over its possible values.
	To do so, it will be useful to introduce 
	\begin{align}\label{eq:defs-MkmkZk}
		M_k :=& \sup_{s\in [k,k+1]} (X_s-bs), \quad 
		 m_k  :=\inf_{s\in [k,k+1]} (X_s-bs), \nonumber \\
		\text{ and } & Z_k :=\sup_{s,s'\in [k,k+1]} |(X_s-bs) - (X_s-bs)|.
	\end{align}
	For $k+1\le T/2$ such that $[k,k+1]\cap \mathtt{Bulk}\ne \emptyset$, it follows from the definition that
	\begin{multline}
		\{\Omega^h\setminus \Omega, \tau\in [k,k+1]\}
		\subset \\
		\{M_k\in [a,a+h(k+1))\}
		\cap
		\{X_s\le bs+M_k \, \forall s\le k\}
		\cap \{X_s\le bs+M_k \, \forall s\in [k+1,T]\}
	\end{multline}
	We will calculate the probability of the event {$E$} on the right-hand side under $\PB{x}{a+bT-j}{T}{d}$ conditioned on $X_k,X_{k+1},M_k$.
	Thanks to the Markov property, conditioned on $X_k,X_{k+1},M_k$, the processes  $\{X_t, t \in [0,k]\},\{X_t, t \in [k+1,T]\}$ are respectively distributed according to $\PB{x}{X_k}{k}{d},\PB{X_{k+1}}{a+bT-j}{T-k-1}{d}$. 
	Hence,
\begin{align}\label{eq:RHS-cond}
		&  \PB{x}{a+bT-j}{T}{d}(E| X_k,M_k,X_{k+1})\\ \nonumber
        & \le
		\mathds{1}_{[M_k\in [a,a+h(k+1)]]}
		\PB{x}{X_k}{k}{d}(\Omega_{M_k,b,k})
		\PB{X_{k+1}}{a+bT-j}{T-k-1}{d}(\Omega_{M_k+b(k+1),b,T-k-1}).
	\end{align}
	Before integrating over $X_k,X_{k+1},M_k$, let us further evaluate each of these terms. 
We write $X_k = M_k + bk - (M_k-(X_k-bk))$, and apply \cref{eq:thm-non-sharp-bound} to both terms, to get
\begin{align}
		\label{eq:bulk-bessel-part-1}
		\PB{x}{X_k}{k}{d}(\Omega_{M_k,b,k})
		 & \lesssim
		\frac{Z_k(M_k-x+1)}{k}, \text{ and}
		\\
		\label{eq:bulk-bessel-part-2}
		\PB{X_{k+1}}{a+bT-j}{T-k-1}{d}(\Omega_{M_k+b(k+1),b,T-k-1})
		 & \lesssim
		\frac{ (j+M_k-a) (Z_k+1) }{T-k-1}.
	\end{align}

	Using that $Z_k$ has Gaussian tails and the explicit law of $X_k$, we can integrate over $ X_k,M_k,$ and $X_{k+1}$ to obtain that 
	\begin{align*}
		\PB{x}{a+bT-j}{T}{d}(E) \lesssim \sum_{k=\sqrt{a-x}}^{T/2}
		\sum_{\ell=0}^{\infty}
		& (\ell+1)^2
		  \frac{(a-x + h(k+1) +1)}{k}\frac{j+h(k+1) }{T-k-1}
		\\
		 & \times \PB{x}{a+bT-j}{T}{d}(Z_k \in [\ell,\ell+1],
		M_k \in [a,a+h(k+1) ]
		)
	\\
		\lesssim
		\frac{((a-x) \vee 1)(j \vee 1)}
		{T}
		\sum_{k=\sqrt{a-x}}^{T/2}
		\sum_{\ell=1}^{\infty}
		 & (\ell+1)^2
		k^{2\gamma-1}                                                                              \\
		 & \times \PB{x}{a+bT-j}{T}{d}\left(X_k \in [a+bk,a+bk+h(k)],Z_k \in [\ell-1,\ell]\right).
	\end{align*}
Finally using that $0<h(t)< c t^{\gamma}$ with $\gamma<1/6$ (by assumption) and choosing $p > 1$ such that $\gamma(2+1/p)<1/(2p)$, we use the H\"older inequality to bound this by a $p,\gamma$ dependent constant times
	\begin{align*}
		\frac{((a-x) \vee 1)(j \vee 1)}{T}
		\sum_{k=\sqrt{a-x}}^{\infty}
		k^{\gamma(2+1/p)-(1+1/(2p))}
		\lesssim
		\frac{((a-x) \vee 1)(j \vee 1)}{T}
		(a-x)^{-\varepsilon_{p,\gamma}}
	\end{align*}
	with $\varepsilon_{p,\gamma}>0$. We proceed in a similar manner to sum over $k$ between $\lfloor T/2 \rfloor$ and $\lfloor T-j^{1/2} \rfloor$, and get
	\begin{equation*}
		\mathcal{L}_{a,j,b,x}^{\mathtt{Bulk}}
		\lesssim
j (a-x)^{1-\varepsilon_{p,\gamma}}
		+
		j^{1-\varepsilon_{p,\gamma}} (a-x),
	\end{equation*}
	concluding \cref{lem:c-bulk-bdy}(i).
\end{proof}

\begin{proof}[Proof of \cref{lem:c-bulk-bdy}(ii)]
	We will have to break this into two extra cases, for which we will consider the following edge intervals
\begin{align*}
		\mathtt{E}^1=\mathtt{E}^1_{a,x,T} := \left[0,(a-x)^{1/2}\right], &
		\quad
		\mathtt{E}^2=\mathtt{E}^2_{j,T}  := \left[T-j^{1/2},T\right].
	\end{align*}
	For $i=1,2$, let
	$
		\mathcal{L}^{i}_{a,b,j,x}
		:=
		\limsup_{T \to \infty}
		T
		\PB{x}{a+bT-j}{T}{d}(\Omega^h\setminus \Omega, \tau \in \mathtt{E}^i).
		$
		We claim that  $\mathcal{L}^{i}_{a,b,j,x}$ is finite and 
		$
		\lim_{\substack{ a,j \to \infty }}
		\mathcal{L}^{i}_{a,b,j,x}/(aj) = 0 $
	 for $i=1,2.$
	Once again, we consider $M_k,m_k,Z_k$ as defined in \cref{eq:defs-MkmkZk}.
	For $[k,k+1] \cap \mathtt{E}^1 \neq \emptyset$, we have that
\begin{align*}
		\PB{x}{a+bT-j}{T}{d}
		 &
		\left(\Omega^h\setminus \Omega, \tau \in [k,k+1] \mid M_{k},  X_{k+1},Z_k\right)
		\\&  \le
		\mathds{I}\left[M_k \in \left(a,a+h_T\left( (a-x)^{1/2}\right)\right)\right]
		\times
		\PB{X_{k+1}}{a+bT-j}{T-k-1}{d}
		\left(\Omega_{M_k+b(k+1),b,T-k-1 }\right)
		\\&  \lesssim
		\mathds{I}\left[M_k \in \left(a,a+h_T\left( (a-x)^{1/2}\right)\right)\right]
		\times
		\frac{Z_k \left(j+h_T\left(\left(a-x\right)^{1/2}\right) \right)}{T}
	\end{align*}
where in the last inequality, we used \cref{eq:bulk-bessel-part-2}.
	Summing over $k \in \mathtt{E}^1$ and taking the expectations over $M_k,X_{k+1},Z_k$, we have
\begin{align*}
		 &
		\sum_{k \in \mathtt{E}^1}
		\PB{x}{a+bT-j}{T}{d}
		\left(\Omega^h\setminus \Omega, \tau \in [k,k+1] \right)
		\\&  \lesssim
		\frac{\left(j+h_T\left(\left(a-x\right)^{1/2}\right) \right)}{T}
		\sum_{k \in \mathtt{E}^1}
		\EB{x}{a+bT-j}{T}{d}
		\left[(Z_k+1) \mathds{I}_{\left[M_k \in \left(a,a+h_T\left( (a-x)^{1/2}\right)\right)\right]} \right]
		\\&  \lesssim
		\frac{(j \vee 1)}{T}
		\left(1 \vee h_T\left(  \left(a-x\right)^{1/2} \right)\right)^{1/2}
		(a-x)^{1/2}
		\lesssim
		\frac{(j \vee 1)}{T} 1 \vee(a-x)^{8/10},
	\end{align*}
	where in the second last inequality we used again that $Z_k$ has Gaussian tails.
	With this, we obtain the desired estimates for $\mathtt{E}^1$.
	The contribution of $\mathtt{E}^2$ is estimated similarly and is left to the reader.
\end{proof}

\bibliographystyle{abbrv}
\bibliography{library}

@article{jego2020planar,
author = {Antoine Jego},
title = {{Planar {B}rownian motion and {G}aussian multiplicative chaos}},
volume = {48},
journal = {The Annals of Probability},
number = {4},
publisher = {Institute of Mathematical Statistics},
pages = {1597--1643},
year = {2020},
doi = {10.1214/19-AOP1399},
URL = {https://doi.org/10.1214/19-AOP1399}
}

@article{bramson2016convergence,
  title={Convergence in law of the maximum of nonlattice branching random walk},
AUTHOR = {Bramson, Maury and Ding, Jian and Zeitouni, Ofer},
  JOURNAL = {Annales de l'Institut Henri Poincar\'e{} Probabilit\'es et
              Statistiques},
  FJOURNAL = {Annales de l'Institut Henri Poincar\'e{} Probabilit\'es et
              Statistiques},
    VOLUME = {52},
      YEAR = {2016},
    NUMBER = {4},
     PAGES = {1897--1924}
}

@article{bramson1978maximal,
  title={Maximal displacement of branching {B}rownian motion},
  author={Bramson, Maury},
  journal={Communications on Pure and Applied Mathematics},
  volume={31},
  number={5},
  pages={531--581},
  year={1978},
  publisher={Wiley Online Library}
}

@article{Addario-Berry_Reed_2008,
	title={Ballot theorems for random walks with finite variance},
	url={http://arxiv.org/abs/0802.2491},
	DOI={10.48550/arXiv.0802.2491},
	number={arXiv:0802.2491},
	publisher={arXiv},
	author={Addario-Berry, L. and Reed, B. A.}, 
	year={2008}, 
}

@misc{lawler2018notes,
  author       = {Lawler, Gregory F.},
  title        = {Notes on the {B}essel Process},
  year         = {2018},
  howpublished = {\url{https://www.math.uchicago.edu/~lawler/bessel18new.pdf}},
  note         = {Lecture notes. Accessed 2025-11-13},
}

@inproceedings{pitman1981bessel,
  title={{B}essel processes and infinitely divisible laws},
  author={Pitman, Jim and Yor, Marc},
  booktitle={Stochastic Integrals: Proceedings of the LMS Durham Symposium, July 7--17, 1980},
  pages={285--370},
  year={1981},
  organization={Springer}
}

@article{pitman1982decomposition,
  title={A decomposition of {B}essel bridges},
  author={Pitman, Jim and Yor, Marc},
  journal={Zeitschrift f{\"u}r Wahrscheinlichkeitstheorie und verwandte Gebiete},
  volume={59},
  pages={425--457},
  year={1982},
  publisher={Springer}
}

@article{salminen2011hitting,
  title={On hitting times of affine boundaries by reflecting {B}rownian motion and {B}essel processes},
  author={Salminen, Paavo and Yor, Marc},
  journal={Periodica Mathematica Hungarica},
  volume={62},
  pages={75--101},
  year={2011},
  publisher={Akad{\'e}miai Kiad{\'o}, co-published with Springer Science+ Business Media BV~…}
}

@article{bednorz2023moments,
  title={Moments and tails of hitting times of {B}essel processes and convolutions of elementary mixtures of exponential distributions},
  author={Bednorz, Witold and {\L}ochowski, Rafa{\L}},
  journal={Bernoulli},
  volume={29},
  number={3},
  pages={2219--2246},
  year={2023},
  publisher={Bernoulli Society for Mathematical Statistics and Probability}
}

@article{kent1980eigenvalue,
  title={Eigenvalue expansions for diffusion hitting times},
  author={Kent, John T},
  journal={Zeitschrift f{\"u}r Wahrscheinlichkeitstheorie und Verwandte Gebiete},
  volume={52},
  number={3},
  pages={309--319},
  year={1980},
  publisher={Springer}
}

@article{roberts2013simple,
author = {Roberts, Matthew},
title = {{A simple path to asymptotics for the frontier of a branching {B}rownian motion}},
volume = {41},
journal = {The Annals of Probability},
number = {5},
publisher = {Institute of Mathematical Statistics},
pages = {3518--3541},
year = {2013},
doi = {10.1214/12-AOP753},
URL = {https://doi.org/10.1214/12-AOP753}
}

@article{bolthausen2001entropic,
  title={Entropic repulsion and the maximum of the two-dimensional harmonic},
  author={Bolthausen, Erwin and Deuschel, Jean-Dominique and Giacomin, Giambattista},
  journal={The Annals of Probability},
  volume={29},
  number={4},
  pages={1670--1692},
  year={2001},
  publisher={Institute of Mathematical Statistics}
}

@article{bramson2012tightness,
  title={Tightness of the recentered maximum of the two-dimensional discrete {G}aussian free field},
  author={Bramson, Maury and Zeitouni, Ofer},
  journal={Communications on Pure and Applied Mathematics},
  volume={65},
  number={1},
  pages={1--20},
  year={2012},
  publisher={Wiley Online Library}
}

@article{takacs1989ballots,
	title={Ballots, queues and random graphs},
	volume={26},
	DOI={10.2307/3214320},
	number={1},
	journal={Journal of Applied Probability},
	author={Takács, Lajos},
	year={1989},
	pages={103–112}
}

@incollection{banderier2019kernel,
author="Banderier, Cyril and Wallner, Michael",
title="The Kernel Method for Lattice Paths Below a Line of Rational Slope",
bookTitle="Lattice Path Combinatorics and Applications",
year="2019",
publisher="Springer International Publishing",
pages="119--154",
isbn="978-3-030-11102-1",
doi="10.1007/978-3-030-11102-1_7",
url="https://doi.org/10.1007/978-3-030-11102-1_7"
}

@book{sznitman2013brownian,
  title={Brownian motion, obstacles and random media},
  author={Sznitman, Alain-Sol},
  year={1998},
  publisher={Springer Science \& Business Media}
}

@misc{DLMF,
  title        = {{NIST} Digital Library of Mathematical Functions},
  author = {Olver, Frank W. J. and Olde Daalhuis, Adri B. and Lozier, Daniel W. and Schneider, Barry I. and Boisvert, Ronald F. and Clark, Charles W. and Miller, Bruce R. and Saunders, Boyana V. and Cohl, Howard S. and McClain, Milton A.},
  year         = {2010},
  howpublished = {\url{https://dlmf.nist.gov/}},
  note         = {Release 1.1.10. Accessed 2025-11-13},
  organization = {National Institute of Standards and Technology},
}

@article{Durbin_1971,
	title={Boundary-crossing probabilities for the {B}rownian motion and {P}oisson processes and techniques for computing the power of the {K}olmogorov-{S}mirnov test},
	volume={8},
	DOI={10.2307/3212169},
	number={3},
	journal={Journal of Applied Probability},
	author={Durbin, J.},
	year={1971},
	pages={431–453}
}

@article{doob1949heuristic,
  title={Heuristic approach to the {K}olmogorov-{S}mirnov theorems},
  author={Doob, Joseph},
  journal={The Annals of Mathematical Statistics},
  pages={393--403},
  year={1949},
  publisher={JSTOR},
  volume={20},
  issue={3}
}
\end{document}